\documentclass[11pt]{amsart}

\usepackage{amsfonts,epsfig}
\usepackage{latexsym}
\usepackage{amssymb}
\usepackage{amsmath}
\usepackage{amsthm}
\usepackage{graphics}
\usepackage[all]{xy}
\usepackage[T2A]{fontenc}
\usepackage{multirow}
\usepackage{hhline}
\usepackage{array, booktabs, ctable}
\usepackage{mathtools}
\usepackage[backref]{hyperref}
\usepackage{enumerate}
\usepackage{comment}
\usepackage{breqn}
\usepackage{lscape}
\usepackage{chngpage}
\usepackage{longtable}
\usepackage{soul}
\usepackage{verbatim}

\newtheorem{defn}{Definition}[section]

\newtheorem{corollary}[defn]{Corollary}
\newtheorem{lemma}[defn]{Lemma}

\newtheorem{prop}[defn]{Proposition}

\newtheorem*{theorem*}{Theorem}

\theoremstyle{definition}

\newtheorem*{ack}{Acknowledgements}
\newtheorem{remark}[defn]{Remark}

\newtheorem{example}[defn]{Example}

\newcommand{\Q}{\mathbb Q}
\newcommand{\Z}{\mathbb Z}

\newcommand{\HH}{\mathbb H}

\newcommand{\SL}{\operatorname{SL}}
\newcommand{\PSL}{\operatorname{PSL}}

\newcommand{\Gal}{\operatorname{Gal}}
\newcommand{\Aut}{\operatorname{Aut}}

\newcommand{\GL}{\operatorname{GL}}

\begin{document}



%\bibliographystyle{plain}
\title[Genus formulas for families of modular curves]{Genus formulas for families of modular curves}

\author{Asimina S. Hamakiotes}
\address{Department of Mathematics, University of Connecticut, Storrs, CT 06269, USA}
\email{asimina.hamakiotes@uconn.edu} 
\urladdr{https://asiminah.github.io/}

\author{Jun Bo Lau}
\address{Department of Electrical Engineering, KU Leuven, Kasteelpark Arenberg 10/2452, 3001 Leuven (Heverlee), Belgium
}
\email{junbo.lau@kuleuven.be} 
\urladdr{}

\thanks{Lau was supported partly by the Simons Foundation grant \#550023 for the Simons Collaboration on Arithmetic Geometry, Number Theory, and Computation, and in part by the European Research Council (ERC) under the European Union’s
Horizon 2020 research and innovation programme (grant agreement ISOCRYPT - No. 101020788), 
by the Research Council KU Leuven grant C14/24/099 and by CyberSecurity Research Flanders with reference 
number VR20192203.}

\newcommand{\asimina}[1]{{\color{blue} \sf Asimina: [#1]}}
\newcommand{\jun}[1]{{\color{red} \sf Jun: [#1]}}

\begin{abstract} 
For each open subgroup $H\leq \GL_2(\widehat{\Z})$, there is a modular curve $X_H$, defined as a quotient of the full modular curve $X(N)$, where $N$ is the level of $H$. The genus formula of a modular curve is well known for $X_0(N)$, $X_1(N)$, $X(N)$, $X_{\mathrm{sp}}(N)$, $X_{\mathrm{ns}}(N)$, and $X_{S_4}(p)$ for $p$ prime. We explicitly work out the invariants of the genus formulas for $X_{\mathrm{sp}}^+(N)$, $X_{\mathrm{ns}}^+(N)$, and $X_{\text{arith},1}(M,MN)$. In Table \ref{tab:all_invariants}, we provide the invariants of the genus formulas for all of the modular curves listed. %A table with the invariants of their genus formulas is available. % for the families of modular curves listed on the LMFDB \cite{lmfdb}.

%For each open subgroup $H\leq \GL_2(\widehat{\Z})$, there is a modular curve $X_H$, defined as a quotient of the full modular curve $X_{\text{full}}(N)$, where $N$ is the level of $H$. The general genus formula of a modular curve is well known and has been worked out explicitly for $X_0(N)$, $X_1(N)$, $X(N)$, $X_{\mathrm{sp}}(N)$, $X_{\mathrm{ns}}(N)$, and $X_{S_4}(p)$; however, the genus formulas for $X_{\mathrm{sp}}^+(N)$, $X_{\mathrm{ns}}^+(N)$, and $X_{\text{arith}}(N)$ cannot be found in the literature. In this paper, we explicitly work out the genus formulas for the families of modular curves listed in the LMFDB and provide a table with the invariants of the genus formulas for all familes of modular curves. This data can also be found on the LMFDB \cite{lmfdb}.

\end{abstract}

%%%%%%%%%%%%%%%%%%%%%%%%%%%%%%%%%%%%%%%%%%%%%%%%%%%%%%%%%%%%%%%%

\maketitle

\section{Introduction}

%A modular curve $X(\Gamma)$ is a quotient of the upper half plane $\HH$ by the action of a congruence subgroup $\Gamma \subseteq \SL_2(\Z)$. 

Let $E$ be an elliptic curve defined over a number field $K$, with algebraic closure $\overline{K}$. For an integer $N\geq 2$, let $E[N]$ be the $N$-torsion subgroup of $E(\overline{K})$. Note that $E[N]$ is a free $\Z/N\Z$-module of rank $2$. The absolute Galois group $\Gal(\overline{K}/K)$ acts on $E[N]$ via the natural action on the coordinates of the points in $E[N]$, which induces the following Galois representation after fixing a basis of $E[N]$:
\[
\rho_{E,N}: \Gal(\overline{K}/K) \rightarrow \Aut(E[N]) \cong \GL_2(\Z/N\Z).
\]

%\begin{theorem}[Serre's Uniformity Conjecture]
%    Let $F$ be a number field and $E$ an elliptic curve over $F$ without complex multiplication. Then there is a constant $C_{E,K}$ such that for all $p > C_{E,K}$ the representation $\rho_{E,p}$ is surjective.   
%\end{theorem}

For a subgroup $H \leq \GL_2(\Z/N\Z)$, the $K$-rational points of the modular curves $X_H(N)$ parametrize elliptic curves $E/K$ such that $\operatorname{im}\rho_{E,N}(\Gal(\overline{K}/K)) \subseteq H$. Classifying rational points on modular curves has been a subject of study for many years, e.g., Serre's Uniformity Problem \cite{serre1972proprietes}, Mazur's ``Program B'' \cite{mazur77rationalpoints}. Faltings Theorem tells us that for 
a smooth curve $C$ with genus $g\geq 2$, there are finitely many rational points on $C$. The genus of a curve $C$ allows us to determine the arithmetic of the rational points on $C$ and helps us decide what method to use to compute rational points. 

%genus determines the "big picture" arithmetic of the rational points, thanks to Faltings' theorem. Also, the genus intervenes in considerations when trying to decide on a method to compute rational points. For instance, the genus is one of the invariants to be considered in the Chabauty methods.

In the classification of possible $2$-adic and $\ell$-adic images of Galois representations attached to elliptic curves over $\Q$, Rouse, Sutherland, and Zureick-Brown developed computational tools to determine the genus of modular curves \cite{rouse2015elliptic,ell_adic}. More precisely, given an integer $N > 1$ and a subgroup type $H \leq \GL_2(\Z/N\Z)$, one could compute the generators of $H$, the index of the subgroup, number of elliptic points of order $2$ and $3$, and the number of cusps.

In this paper, we determine elementary genus formulas of modular curves associated to subgroups of $\GL_2(\Z/N\Z)$ such that they admit a model over $\Q$, without group-theoretic inputs. The list includes various families of interest which can be found on \cite{lmfdb} and contains the list from \cite{serre1972proprietes}. Below is a summary of the families of modular curves for which the genus formula is known, with references where one can find the genus formulas.  
%The modular curves $X_0(N)$, $X_1(N)$, and $X(N)$ are of importance in the theory of classical modular forms. The remaining families of modular curves, $X_{\mathrm{sp}}(N)$, $X_{\mathrm{sp}}^+(N)$, $X_{\mathrm{ns}}(N)$, $X_{\mathrm{ns}}^+(N)$, $X_{S_4}(p)$, $X_{\mathrm{arith}}(N)$, are of particular interest and importance in arithmetic geometry.  
%- We are focusing on families of congruence subgroups that are of importance in arithmetic geometry. 
%- General genus formulas for families of modular curves that play a particularly important role in arithmetic geometry. 
%If you want to reference why these ones are important, you can reference Serre's classification of the max subgps of GL2fp and how many of these come from that classificaiton (mod p that's where many of these come from). X0, X1, X come from the theory of modular forms. The other ones come from Serre's classification of the mod p maximal subgroups from his open image theorem paper. 
%Below is a summary of prior results. 
\begin{enumerate}
    \item Let $N\geq 1$ be an integer. For the families $X_0(N), X_1(N)$, and $X(N)$, one can find the genus formulas worked out in \cite{Diamond2005,shimura1971}, among many other sources.

    \item Let $p$ be a prime and $r\geq 1$ an integer. In \cite[Table 3.1]{DLM2022} and references therein, one can find the invariants of the genus formulas for $X_{\mathrm{sp}}(p^r)$, $X_{\mathrm{sp}}^*(p^r)$, $X_{\mathrm{ns}}(p^r)$, and $X_{\mathrm{ns}}^*(p^r)$ (see Remark \ref{rem:plus_star} for details about the $X^*$ notation). Since the genus formula invariants are multiplicative \cite[Equation 3.14]{DLM2022}, one can determine the invariants of the genus formula for composite $N$ for $X_{\mathrm{sp}}(N)$, $X_{\mathrm{sp}}^*(N)$, $X_{\mathrm{ns}}(N)$, and $X_{\mathrm{ns}}^*(N)$, as given in Table \ref{tab:all_invariants} of this paper.

    \item Let $p$ be a prime such that $p \equiv \pm 3 \bmod 8$. In \cite{Ligozat}, Ligozat determined the genus formula of $X_{S_4}(p)$. Also, see \cite{serre1972proprietes,swinntertondyer73}.
\end{enumerate}
Note that in various literature, the notations $X^*$ and $X^+$ are used interchangeably but the modular curves and corresponding subgroups of $\GL_2(\Z/N\Z)$ are different (see Remark \ref{rem:plus_star}). 

The results of this paper are summarized in Table \ref{tab:all_invariants}, where we list the invariants of the genus formulas of all the modular curves discussed in this paper. In addition, the results of this work are published on the L-functions and Modular Forms Database (LMFDB) in the ``Modular curves over $\Q$'' section \cite{lmfdb}. There are pages for each family of modular curves in the database, where the invariants of the genus formula for that family are displayed. The methods used in this paper follow a general strategy that can be found in standard texts on modular forms \cite{Diamond2005,shimura1971}.

%\asimina{also mention that this work will be "published"/can be found on the lmfdb.}

\subsection{Structure of the paper} %In this paper, we determine the genus formulas for composite level $N$ for the following families of modular curves:
In Section \ref{sec:background}, we review background relating to modular curves and the genus formula of a modular curve. %In Section \ref{sect:outline}, we give an outline of the general method used to determine the genus formula. 
In Section \ref{sec:split_cartan}, we determine the genus formulas for composite level $N$ for $X_{\mathrm{sp}}(N)$, $X_{\mathrm{sp}}^*(N)$, $X_{\mathrm{sp}}^+(N)$. In Section \ref{sec:nonsplit_cartan}, we determine the genus formulas for composite level $N$ for $X_{\mathrm{ns}}(N)$, $X_{\mathrm{ns}}^*(N)$, $X_{\mathrm{ns}}^+(N)$. In Section \ref{Xarith1}, we determine the genus formula for $X_{\mathrm{arith},1}(M,MN)$ and $X_{\mathrm{arith},\pm 1}(M,MN)$.

%\subsection{Structure of paper} In Section \ref{sec:background}, we review background and preliminary results relating to the genus formula. In Section \ref{sec:split_cartan}, we will discuss and determine the genus formula for $X_{\mathrm{sp}}(N)$, $X_{\mathrm{sp}}^*(N)$, and $X_{\mathrm{sp}}^+(N)$. In Section \ref{sec:nonsplit_cartan}, we will discuss and determine the genus formula for $X_{\mathrm{ns}}(N)$, $X_{\mathrm{ns}}^*(N)$, and $X_{\mathrm{ns}}^+(N)$. 

%\subsection{Code} If we decide to add any \verb|Magma| \cite{Magma} code to GitHub to verify computations we can include the github link here here. \asimina{I think we should post the code you used to verify the formulas on github or something so that in the paper we can say ``verified by a Magma computation" or something and reference it.}

\begin{ack}
    The authors would like to thank Jennifer Balakrishnan, Harris Daniels, Jeremy Rouse, Andrew Sutherland, John Voight, and David Zywina for help with references. They would like to thank Pietro Mercuri for sharing his work on split Cartan subgroups. They would also like to thank the Simons Collaboration on Arithmetic Geometry, Number Theory, and Computation for funding the Modular Curves 3 Workshop hosted by MIT, where this work began. 
\end{ack}

%%%%%%%%%%%%%%%%%%%%%%%%%%%%%%%%%%%%%%%%%%%%%%%%%%%%%%%%%%%%%%%%

\section{Background}\label{sec:background}

Let $N$ be a positive integer. There is a functor sending a $\Z[1/N]$-algebra $R$ to the set of (isomorphism classes of) pairs $(E,\phi)$ where $E$ is an elliptic curve over $R$ and $\phi: E[N]\rightarrow (\Z/N\Z)^2$ is an isomorphism of $R$-group schemes. If $N \geq 3$, then the functor is representable by a smooth affine $\Z[1/N]$-scheme, denoted by $Y(N)$. If $N <3$, we take the coarse moduli space to get a scheme. We denote by $X(N)$ the compactification of $Y(N)$ and call this the full modular curve of level $N$.

Every matrix $\gamma \in \GL_2(\Z/N\Z)$ is an automorphism of $(\Z/N\Z)^2_R$, with $\gamma$ acting on $(E,\phi)$ by sending it to $(E,\phi \circ \gamma)$. For a subgroup $H \leq \GL_2(\Z/N\Z)$, we denote by $X_H$ the quotient $X(N)/H$. As a coarse moduli space, $Y_H := Y(N)/H$ parametrizes elliptic curves with $H$-structure. More precisely, the equivalence class of elliptic curves with $H$-structure is given by $(E,\phi)$ where $E$ is an elliptic curve over a $\Z[1/N]$-scheme $R$ and $\phi:E[N] \rightarrow (\Z/N\Z)^2$ is an isomorphism of $R$-group schemes and $(E,\phi) \sim_H (E',\phi')$ if and only if $(\phi')^{-1} \circ \iota|_{E[N]} \circ \phi = h$ for some $h \in H$ and $\iota : E \xrightarrow[]{\sim} E'$.

On the other end, it is well known that for a subgroup $\Gamma \leq \SL_2(\Z)$ with finite index, one could define a complex structure on $\HH^*/\Gamma$, where $\HH^*$ denotes the union of the upper half plane and the cusps of $\Gamma$. In particular, $X_H$ defined above can be realized as a quotient of the upper half plane $\HH^*/\Gamma_H$ where $\HH^*$ is the union of the upper half plane and the point at infinity, and $\Gamma_H \leq \SL_2(\Z)$ is defined as the lift of $H \cap \SL_2(\Z/N\Z)$. 

Before stating the genus formula of a modular curve $X_H$, we will define the following invariants: 
\begin{itemize}
    \item The $\PSL_2$\textit{-index} of $X_H$, denoted $i(\Gamma_H)$, is the index of $\bar{H} \coloneqq \pm H \cap \SL_2(\Z/N\Z)$ in $\PSL_2(\Z/N\Z)$.  
    
    \item The \textit{number of elliptic points of order} 2 on $X_H$, denoted $\varepsilon_2(\Gamma_H)$, is the number of right cosets of $\Gamma_H$ in $\SL_2(\Z/N\Z)$ fixed by the right action of $\begin{psmallmatrix}0&-1\\1&0\end{psmallmatrix}$.

    \item The \textit{number of elliptic points of order} 3 on $X_H$, denoted $\varepsilon_3(\Gamma_H)$, is the number of right cosets of $\Gamma_H$ in $\SL_2(\Z/N\Z)$ fixed by the right action of $\begin{psmallmatrix}0&-1\\1&-1\end{psmallmatrix}$.

    \item The \textit{number of cusps} on $X_H$, denoted $\varepsilon_\infty(\Gamma_H)$, is the number of orbits of the right coset space $\Gamma_H \backslash \SL_2(\Z/N\Z)$ under the right action of $\begin{psmallmatrix}1&1\\0&1\end{psmallmatrix}$.
\end{itemize}

The Riemann-Hurwitz formula imples that the genus of the modular curve $X_H$ is given by:
\[
g(\Gamma_H) = 1 + \frac{i(\Gamma_H)}{12} -\frac{\varepsilon_2(\Gamma_H)}{4} - \frac{\varepsilon_3(\Gamma_H)}{3} - \frac{\varepsilon_\infty(\Gamma_H)}{2}.
\]
Common notation used in various genus formulas is: $\varphi(N)$ which is the usual Euler totient function, $\omega(N)$ which is the number of primes dividing $N$, and $\left(\frac{\cdot}{\cdot}\right)$ which is the usual Legendre symbol. 

\begin{remark}
Let $H \leq \GL_2(\Z/N\Z)$. If $-I \in H$, then $H$ is a coarse group. Otherwise, $H$ is a fine group. An index 2 or \textit{quadratic refinement} of $H$ is a subgroup $H' \leq H$ such that $H = \pm H'$. The modular curves $X_H$ and $X_{H'}$ are isomorphic as curves, but the moduli problem is refined from $H$ to $H'$. Naturally, the genus of $X_H$ and $X_{H'}$ are equal. In this paper, the genus formulas for $X_{\pm 1}(N)$, $X_{\mathrm{arith},\pm 1}(M,MN)$, and their quadratic refinements are included in Table \ref{tab:all_invariants}.
\end{remark}

%Fix level N, how do you get a congruence subgroup out of that. Can list X0(N), X1(N), X(N).  

%Outline:  
%\begin{itemize}
    %\item Start with the formal definition of a modular curve (create it as a compactification of Y of gamma)
    %\item Start with a gamma and create the modular curve and then talk about why we're interested in the genus of a modular curve
    %\item Add definitions of elliptic points, cusps, and things like that
    %\item State genus formula
    %\item State Riemann-Hurwitz as a thm 
    %\item Add remark about the difference between $*$ and $+$ notation in other literature. 
%\end{itemize}

\subsection{Modular curves associated to (extended) Cartan subgroups}\label{rem:Cartan_subgps}
    Let $R$ be a commutative ring. Given a free rank 2 etal\'{e} $R$-algebra $A$, for $a \in A^\times$, the multiplication-by-$a$ map defines an injective homomorphism $A^\times \hookrightarrow \GL_2(R)$. The image of this homomorphism is called a \textit{Cartan subgroup} of $\GL_2(R)$. There is a canonical involution of $A$ which gives another element of $\GL_2(R)$. The group generated by the image of $A^\times$ and the involution is called the \textit{extended Cartan subgroup}. In particular, when $R = \Z/N\Z$, this group is called the \textit{normalizer of Cartan subgroup} and it is denoted by either $C_{\mathrm{sp}}^+(N)$ or $C_{\mathrm{ns}}^+(N)$, depending on whether $A$ is split or nonsplit respectively at each prime dividing $N$. The Cartan subgroup has index $2$ in the extended Cartan subgroup. In this paper, the normalizer of the Cartan subgroup will mean the extended Cartan subgroup.

\begin{remark}\label{rem:plus_star}
    In the literature, sometimes the normalizer of the Cartan subgroup is defined as the group generated by the Cartan subgroup along with \textit{an involution for each prime dividing $N$}, denoted by $C_{\mathrm{sp}}^*$ or $C_{\mathrm{ns}}^*$, depending on whether $A$ is split or nonsplit respectively at each prime dividing $N$. In contrast with the previous paragraph, the Cartan subgroup here has index $2^{\omega(N)}$ in the group and is generated by the Cartan subgroup and involutions from each prime dividing $N$.

\end{remark}

Note that Remark \ref{rem:plus_star} also applies to modular curves $X_0(N)$ defined by Atkin-Lehner involutions: $X_0^+(N)$ and $X_0^*(N)$.

\begin{example}
    Consider the modular curve associated to the normalizer of a nonsplit Cartan subgroup of level $39$, $X_{\mathrm{ns}}^+(39)$. Following our notation, $X_{\mathrm{ns}}^+(39)$ correponds to the LMFDB label $\verb|39.468.28.d.1|$ with genus $28$. In \cite{BARAN20102753}, the curve with the same notation is defined as in Remark \ref{rem:plus_star} and has LMFDB label $\verb|39.234.13.a.1|$ and genus $13$.
\end{example}

% \jun{this subsection can be subsumed under the intro bullets?}
\subsection{Modular curves associated to exceptional groups}\label{rem:X_S4}

For $p$ an odd prime, $X_{S_4}(p)$ is the modular curve $X_H$ for $H\leq \GL_2(\widehat{\Z})$ given by the inverse image of the subgroup of $\operatorname{PGL}_2(\Z/p\Z)$ isomorphic to $S_4$ (which is unique up to conjugacy). It parametrizes elliptic curves whose mod-$p$ Galois representation has projective image $S_4$, one of the three exceptional groups $A_4$, $A_5$, $S_4$ of $\operatorname{PGL}_2(p)$ that can arise as projective mod-$p$ images, and the only one that can arise for elliptic curves over $\Q$. The subgroup $H$ contains $-I$ and has surjective determinant when $p \equiv \pm 3 \bmod 8$, but not otherwise (see \cite{serre1972proprietes,swinntertondyer73}).

\subsection{Modular curves with full level structure }
    For $N\geq 1$ an integer, $X_{\mathrm{arith}}(N)$ is the modular curve $X_H$ for $H\leq \GL_2(\widehat{\Z})$ given by the inverse image of $\begin{psmallmatrix}1&0\\0&*\end{psmallmatrix} \leq \GL_2(\Z/N\Z)$. As a moduli space, $X_{\mathrm{arith}}$ parametrizes isomorphism classes of triples $(E,\phi,P)$, where $E$ is a generalized elliptic curve, $P$ is a point of exact order $N$, and $\phi \colon E \to E'$ is a cyclic $N$-isogeny such that $E[N]$ is generated by $P$ and $\ker \phi$. Alternatively, it parametrizes isomorphism classes of pairs $(E,\psi)$, where $E$ is a generalized elliptic curve and $\psi \colon \mu_N \times \Z/N\Z \to E[N]$ is a symplectic isomorphism. 

Note that $X_{\mathrm{arith}}(N)$ corresponds to the subgroup of matrices that intersects $\SL_2(\Z/N\Z)$ trivially and is a connected component of $X(N)$, 
\[
\begin{pmatrix}1&0\\0&*\end{pmatrix} \cap \SL_2(\Z/N\Z) = \begin{pmatrix}1&0\\0&1\end{pmatrix}. 
\]
The modular curve that this subgroup defines coincides with $X(N)$ as Riemann surfaces. Therefore, the genus formula for $X_{\mathrm{arith}}(N)$ is the same as the genus formula for $X(N)$.

\subsection{Modular curves with interesting torsion data}\label{subsec:Xarith1}
Let $E$ be an elliptic curve defined over a number field $K$ of degree $d$. It is known that the torsion subgroup $\#E_{\mathrm{tors}}(K)$ can be generated by two elements and is uniformly bounded in the degree of $K$ (see \cite{derickx2017torsion} and references therein). The modular curve $X_{\mathrm{arith},1}(M,MN)$ is defined by the inverse image of $\begin{psmallmatrix}1&M*\\0&*\end{psmallmatrix} \subset \GL_2(\Z/MN\Z)$. As a moduli space it parametrizes triples $(E,P,C)$, where $E$ is an elliptic curve over $k$, $P\in E[MN](k)$ is a point of order $MN$, and $C\leq E[M](k)$ is a $\Gal_k$-stable cyclic subgroup of order $M$ such that $E[M] = \langle NP \rangle + C$. The modular curve $X_{\mathrm{arith},1}(M,MN)$ is one of the quadratic refinements of $X_{\mathrm{arith},\pm 1}(M,MN)$, which is defined by the inverse image of $\pm \begin{psmallmatrix}1&M*\\0&*\end{psmallmatrix} \subset \GL_2(\Z/MN\Z)$.

%%%%%%%%%%%%%%%%%%%%%%%%%%%%%%%%%%%%%%%%%%%%%%%%%%%%%%%%%%%%%%%%

%\section{Notes}

%Let $\Gamma \leq \SL_2(\Ints)$ be a congruence subgroup and let $\overline{\Gamma}$ be its image in $\PSL_2(\Ints)$. It is well-known \cite{shimura1971,Diamond2005} that, using the Riemann-Hurwitz formula, the genus of the associated modular curve $X(\Gamma) := \Gamma \backslash \mathbb{H}^\ast$ is given by:
%\[
%g(X(\Gamma)) = 1 + \frac{[\PSL_2(\Ints): \overline{\Gamma}]}{12} - \frac{\varepsilon_2}{4} - \frac{\varepsilon_3}{3} - \frac{\varepsilon_\infty}{2}
%\]
%where $\varepsilon_2,\varepsilon_3$ denotes the number of elliptic points of orders $2$ and $3$ respectively and $\varepsilon_\infty$ denotes the number of inequivalent cusps.

%\begin{itemize}
%    \item Here we can introduce the general genus formula, maybe define elliptic points, cusps, etc.? We can also talk about how this information can be found on the LMFDB \cite{lmfdb}.

%    \item We can say which genus formulas are well known and where one can read about them. 

%    \item We can refer to our table with all of the genus formulas, and then say in which sections we work each formula out. 
%\end{itemize}

%%%%%%%%%%%%%%%%%%%%%%%%%%%%%%%%%%%%%%%%%%%%%%%%%%%%%%%%%%

\section{Split Cartan and normalizer of the split Cartan}\label{sec:split_cartan}

In this section, we further describe the modular curves $X_{\mathrm{sp}}(N)$, $X_{\mathrm{sp}}^+(N)$, and $X_{\mathrm{sp}}^*(N)$ from Section \ref{rem:Cartan_subgps} and Remark \ref{rem:plus_star}. We derive the invariants of their genus formulas for composite level $N$, using methods communicated by Pietro Mercuri for prime power levels $p^r$ in \cite[Remark 4.3]{dose2019double}.

%For this I think we can use references in \cite{rebolledo2017moduli} and \cite{DLM2022}.

%\asimina{The definitions/descriptions that I am using below are from the LMFDB, but feel free to edit them if you want.}

The split Cartan, denoted $X_{\mathrm{sp}}(N)$, is the modular curve $X_H$ for the subgroup $H\leq \GL_2(\widehat{\Z})$ given by the inverse image of a Cartan subgroup $\begin{psmallmatrix}*&0\\0&*\end{psmallmatrix}$ that is split at every prime dividing $N$. As a moduli space it parametrizes triples $(E,C,D)$, where $E$ is an elliptic curve over $k$, and $C$ and $D$ are $\Gal_k$-stable cyclic subgroups such that $E[N](\overline{k}) \simeq C \oplus D$. 

The normalizer of the split Cartan, denoted $X_{\mathrm{sp}}^+(N)$, is the modular curve $X_H$ for the subgroup $H\leq \GL_2(\widehat{\Z})$ given by the inverse image of an extended Cartan subgroup $\begin{psmallmatrix}*&0\\0&*\end{psmallmatrix} \cup \begin{psmallmatrix}0&*\\ *&0\end{psmallmatrix} \subseteq \GL_2(\Z/N\Z)$ that is split at every prime dividing $N$. As a moduli space it parametrizes pairs $(E,\{C,D\})$, where $E$ is an elliptic curve over $k$, and $\{C,D\}$ is a $\Gal_k$-stable pair of cyclic subgroups such that $E[N](\overline{k}) \simeq C \oplus D$. Note that neither $C$ nor $D$ need be $\Gal_k$-stable. 

%\asimina{Split Cartan * stuff below, add the set-up above here.}

From the definition above, the orders of the split Cartan subgroup $C_{\mathrm{sp}}(N)$ and its normalizer $C_{\mathrm{sp}}^+(N)$ are $\varphi(N)^2$ and $2\cdot \varphi(N)^2$, respectively. Therefore, the indices are given by
\begin{align*}
    [\GL_2(\Z/N\Z):C_{\mathrm{sp}}(N)] &= N^2 \cdot \prod_{\substack{p \mid N, \\ p \text{ prime}}} \left(1 + \frac{1}{p}\right), \\
    [\GL_2(\Z/N\Z):C_{\mathrm{sp}}^+(N)] &= \frac{N^2}{2} \cdot \prod_{\substack{p \mid N, \\ p \text{ prime}}} \left(1 + \frac{1}{p}\right).
\end{align*}

\begin{lemma}\label{lem:Csp_cosets}
For prime power levels $N = p^r$, the coset representatives $\GL_2(\Z/N\Z)/C_{\mathrm{sp}}(N)$ have the form:
\[
\alpha(u,v) = \begin{pmatrix}
    1 + uv & u \\ v & 1
\end{pmatrix}, \ \beta(u,v) = \begin{pmatrix}
    u & -1 \\ 1 & pv
\end{pmatrix},
\]
where $u,v \in \Z/p^r\Z$. The coset representatives $\GL_2(\Z/N\Z)/C_{\mathrm{sp}}^+(N)$ are the same as the ones above under the following identifications:
\[
\alpha(u,v) \sim \begin{cases}
    \alpha((1 + uv)v^{-1},-v) & \text{if $v$ is invertible mod $p^r$}, \\
    \beta(u, -\frac{v}{p} (1 + uv)^{-1}) & \text{if $v$ is not invertible mod $p^r$.}
\end{cases}
\]

\end{lemma}

\begin{proof}
    One checks that the number of coset representatives is equal to the $\GL_2$-index and that $\alpha(u',v')^{-1} \alpha(u,v), \beta(u',v')^{-1} \beta(u,v), \alpha(u',v')^{-1}\beta(u,v) \in C_{\mathrm{sp}}(N)$ or $C_{\mathrm{sp}}(N)$ if and only if $u' \equiv u \pmod{N}$ and $v' \equiv v \pmod{N}$. 
\end{proof}

The Chinese Remainder Theorem implies that there is a bijection:
\[
\GL_2(\Z/N\Z)/C_{\mathrm{sp}}(N) \rightarrow \prod_{p} \GL_2(\Z/p_i^{e_i}\Z)/C_{\mathrm{sp}}(p_i^{e_i}).
\]
It follows that the coset representatives on the left are lifts of tuples of $\alpha$'s and $\beta$'s on the right. Let $\omega := \begin{psmallmatrix}
    0 & 1 \\ 1 & 0
\end{psmallmatrix} \in \GL_2(\Z/N\Z)$ be the involution of the split Cartan subgroup. Then the subgroup generated by $ \omega$ acts on $\GL_2(\Z/N\Z)/C_{\mathrm{sp}}(N)$ by left multiplication, which induces a natural action of $\vec{\omega}:= (\omega,\omega, \ldots, \omega)$ on $\prod_{p} \GL_2(\Z/p_i^{e_i}\Z)/C_{\mathrm{sp}}(p_i^{e_i})$, giving the following bijection:
\[
\GL_2(\Z/N\Z)/C_{\mathrm{sp}}^+(N) = \Big(\GL_2(\Z/N\Z)/C_{\mathrm{sp}}(N)\Big)/\langle \omega \rangle \rightarrow \Big(\prod_{p} \GL_2(\Z/p_i^{e_i}\Z)/C_{\mathrm{sp}}(p_i^{e_i})\Big)/\langle \vec{\omega} \rangle.
\]
The identification in Lemma \ref{lem:Csp_cosets} extends to the above bijection, thereby identifying pairs of tuples consisting of $\alpha$'s and $\beta$'s for each prime dividing the level.

With this setup, we can begin calculating the quantities $\varepsilon_2, \varepsilon_3$, and $\varepsilon_\infty$ in the genus formula.

\begin{prop}\label{prop:norm_sp_infty}
    Let $\varepsilon_\infty$ and $\varepsilon_\infty^+$ denote the number of cusps in $X_{\mathrm{sp}}(N)$ and $X_{\mathrm{sp}}^+(N)$, respectively. Then we have that
    \begin{align*}
        \varepsilon_\infty &= N\cdot \prod_{p\mid N, \text{ prime}} \left(1 + \frac{1}{p}\right), \\
        \varepsilon_\infty^+ &= \varepsilon_\infty/2.
    \end{align*}
\end{prop}

\begin{proof}
    Note that $\SL_2(\Z)_\infty = \langle \begin{psmallmatrix}
        1 & 1 \\ 0 & 1
    \end{psmallmatrix}\rangle$. Let $g \in \SL_2(\Z)/\Gamma_{\mathrm{sp}}(N)$ be such that $\begin{psmallmatrix}
        1 & a \\ 0 & 1
    \end{psmallmatrix} \in g^{-1}\Gamma_{\mathrm{sp}}(N)g$. Observe that $a \equiv 0 \pmod{N}$. By \cite[Prop 1.37]{shimura1971}, the ramification index for each cusp is $N$ and since the map $X_{\mathrm{sp}}(N) \rightarrow X(1)$ has degree $N^2 \cdot \prod_p (1 + 1/p)$, we have $\varepsilon_\infty = N \cdot \prod_p (1 + 1/p)$. Furthermore, the morphism $X_{\mathrm{sp}}(N) \rightarrow X_{\mathrm{sp}}^+(N)$ is of degree $2$ and unramified over the cusps. It follows that $\varepsilon_\infty^+ = \varepsilon_\infty/2$. For $ N = 2$, one could work directly with the definition of $\varepsilon_2^+$ as the number of cosets in $\SL_2(\Z/2\Z)$ fixed by the right action of  $\begin{psmallmatrix}
        0 & -1 \\ 1 & 0 
    \end{psmallmatrix}$ and find that there are $2$ cusps. 
\end{proof}

\begin{prop}\label{prop:norm_sp_nu3}
    Let $\varepsilon_3$ and $\varepsilon_3^+$ denote the number of elliptic points of order $3$ in $X_{\mathrm{sp}}(N)$ and $X_{\mathrm{sp}}^+(N)$, respectively. Then,
    % \begin{align*}
    %     \varepsilon_3 &= \begin{cases}
    %        2^{\omega(N)} &p \equiv 1 \pmod{3}  \text{ for all } \ p\mid N, \\
    %        0 & \text{otherwise,}
    %     \end{cases} \\
    %     \varepsilon_3^+ &= \varepsilon_3/2.
    % \end{align*}
    \begin{align*}
    \varepsilon_3 &= \begin{cases}
        0 & \text{if $2\mid N$ \text{or} $3 \mid N$}, \\ \prod\limits_{p\mid N} \left(1 + \left(\frac{-3}{p}\right) \right) & \text{otherwise.}
    \end{cases} \\
    \varepsilon_3^+ &= \varepsilon_3/2.
\end{align*}

\end{prop}

\begin{proof}
    Let $\rho = e^{2\pi i/3}$ be a third root of unity. Note that $\SL_2(\Z)_\rho = \left\langle \begin{psmallmatrix}
        0 & -1 \\ 1 & -1
    \end{psmallmatrix}\right\rangle$. By \cite[Prop 1.37]{shimura1971}, the elliptic points of order 3 on $X_{\mathrm{sp}}(N)$ are the points in the inverse image $f^{-1}(\rho)$, where $f: \HH^*/\Gamma_{\mathrm{sp}}(N) \rightarrow \HH^*/\SL_2(\Z)$, with ramification index 1. By the same proposition, there exists $g \in \SL_2(\Z)/\Gamma_{\mathrm{sp}}(N)$ such that $\SL_2(\Z)_\rho \subseteq g^{-1}\Gamma_{\mathrm{sp}}(N) g$. Therefore, it is sufficient to find the number of coset representatives $\gamma $ such that $\gamma^{-1}\begin{psmallmatrix}
        0 & -1 \\ 1 & -1
    \end{psmallmatrix} \gamma \in C_{\mathrm{sp}}(N)$ and $C_{\mathrm{sp}}^+(N)$, respectively.

    We begin with prime power levels $N = p^r$. We have the following matrices:
    \begin{align}
        \alpha(u,v)^{-1}\begin{psmallmatrix}
        0 & -1 \\ 1 & -1
    \end{psmallmatrix}\alpha(u,v) &= \begin{pmatrix}
        -u^2v + uv - u - v & -u^2 + u - 1 \\ u^2v^2 -uv^2 + 2uv + v^2 -v + 1 & u^2 -uv + u + v -1
    \end{pmatrix}, \label{eq:alpha_conj}\\
    \beta(u,v)^{-1}\begin{psmallmatrix}
        0 & -1 \\ 1 & -1
    \end{psmallmatrix}\beta(u,v) &= (1 + puv)^{-1}\begin{pmatrix}
        u - pv - 1 & -p^2v^2 -pv - 1 \\ u^2 - u + 1 & -puv - u + pv
    \end{pmatrix}. \label{eq:beta_conj}
    \end{align}

    For Equation (\ref{eq:alpha_conj}) to be an element of $C_{\mathrm{sp}}(N)$, we require that the following conditions hold:
\begin{align}
    -u^2 +u -1 &\equiv 0 \pmod{N}, \label{eq:nu3_1} \\
    -v^2(-u^2 +u -1) + 2uv + -v + 1 &\equiv 0 \pmod{N}. \label{eq:nu3_2}
\end{align}

Equation (\ref{eq:nu3_1}) has two solutions if $p \equiv 1 \pmod{3}$ and $0$ otherwise. Equation (\ref{eq:nu3_2}) is determined by Equation (\ref{eq:nu3_1}). For Equation (\ref{eq:beta_conj}) to be an element of $C_{\mathrm{sp}}(N)$, we require that:
\begin{align}
    -p^2v^2 -pv - 1 &\equiv 0 \pmod{N}, \label{eq:nu3_3}\\
    u^2 - u + 1 &\equiv 0 \pmod{N}. \label{eq:nu3_4}
\end{align}
However, this does not yield any solution since $p$ is not invertible modulo $N$ in Equation (\ref{eq:nu3_3}).

Writing the level $N$ as a product of primes, we have two solutions for each prime $p$ dividing $N$ satisfying $p \equiv 
1 \pmod{3}$ and this gives $\varepsilon_3$. By the same argument in calculating $\varepsilon_\infty^+$, the morphism $X_{\mathrm{sp}}(N) \rightarrow X_{\mathrm{sp}}^+(N)$ is of degree $2$ and unramified over the elliptic points of order $3$ and therefore, we have $\varepsilon_3^+ = \varepsilon_3/2$.
\end{proof}

\begin{prop}\label{prop:norm_sp_nu2}
    Let $\varepsilon_2$ and $\varepsilon_2^+$ denote the number of elliptic points of order $2$ in $X_{\mathrm{sp}}(N)$ and $X_{\mathrm{sp}}^+(N)$, respectively. Then we have that 
    \begin{align*}
        \varepsilon_2 &= \begin{cases}
            0 & \text{if $2\mid N$, } \\
            \prod\limits_{p\mid N} \left( 1 + \left(\frac{-1}{p}\right) \right) & \text{otherwise.}
        \end{cases}\\
    \varepsilon_2^+ &= \frac{\varepsilon_2}{2}
        + \left( \frac{N}{2}\cdot \prod_{\substack{p\mid N, \\ p \equiv 1 \bmod 4}}\left( 1 - \frac{1}{p}\right) \cdot \prod_{\substack{p\mid N, \\ p \equiv 3 \bmod 4}}\left(1 + \frac{1}{p}\right)\right),
    \end{align*}
\end{prop}

\begin{proof}
    Note that $\SL_2(\Z)_i = \left\langle \begin{psmallmatrix}
        0 & -1 \\ 1 & 0
    \end{psmallmatrix}\right\rangle$. By \cite[Prop 1.37]{shimura1971}, the elliptic points of order 2 on $X_{\mathrm{sp}}(N)$ are the points in the inverse image $f^{-1}(i)$, where $f: \HH^*/\Gamma_{\mathrm{sp}}(N) \rightarrow \HH^*/\SL_2(\Z)$, with ramification index 1. By the same proposition, there exists $g \in \SL_2(\Z)/\Gamma_{\mathrm{sp}}(N)$ such that $\SL_2(\Z)_i \subseteq g^{-1}\Gamma_{\mathrm{sp}}(N) g$. Therefore, it is sufficient to find the number of coset representatives $\gamma $ such that $\gamma^{-1}\begin{psmallmatrix}
        0 & -1 \\ 1 & 0
    \end{psmallmatrix} \gamma \in C_{\mathrm{sp}}(N)$ and $C_{\mathrm{sp}}^+(N)$, respectively.

    We begin with prime power levels $N = p^r$. We have the following matrices:
    \begin{align}
        \alpha(u,v)^{-1}\begin{psmallmatrix}
        0 & -1 \\ 1 & 0
    \end{psmallmatrix}\alpha(u,v) &= \begin{pmatrix}
    -u^2v -u -v & -u^2 - 1 \\
    u^2v^2 + 2uv + v^2 + 1 & u^2v + u + v
    \end{pmatrix}, \label{eq:alpha_conj2}\\
    \beta(u,v)^{-1}\begin{psmallmatrix}
        0 & -1 \\ 1 & 0
    \end{psmallmatrix}\beta(u,v) &= (1 + puv)^{-1}\begin{pmatrix}
    u-pv & -p^2v^2 -1 \\ u^2 + 1& pv-u
    \end{pmatrix}. \label{eq:beta_conj2}
    \end{align}

For Equation (\ref{eq:alpha_conj2}) to be an element of $C_{\mathrm{sp}}(N)$, we require that the following conditions hold:
\begin{align}
    -u^2 - 1 &\equiv 0 \pmod{N}, \label{eq:nu2_1} \\
    -v^2 (-u^2 -1) + 2uv +  1 &\equiv 0 \pmod{N}. \label{eq:nu2_2}
\end{align}

Equation (\ref{eq:nu2_1}) has two solutions when $p \equiv 1 \pmod{4}$ and $0$ otherwise. Equation(\ref{eq:nu2_2}) is determined by Equation (\ref{eq:nu2_1}). For Equation (\ref{eq:beta_conj2}) to be an element of $C_{\mathrm{sp}}(N)$, we require that:
\begin{align}
    -p^2v^2 -1 \equiv 0 \pmod{N}, \label{eq:nu2_3} \\
    u^2 + 1 \equiv 0 \pmod{N}. \label{eq:nu2_4}
\end{align}
However, this does not yield any solution since $p$ is not invertible modulo $N$ in Equation (\ref{eq:nu2_3}). Writing the composite level $N$ as a product of prime powers, we obtain $\varepsilon_2$.

For $\varepsilon_2^+$, we use a different argument, starting with the prime power level $N = p^r$. For Equations (\ref{eq:alpha_conj2}, \ref{eq:beta_conj2}) to be elements of $C_{\mathrm{sp}}^+(N)$, in addition to Equations (\ref{eq:nu2_1}, \ref{eq:nu2_2}, \ref{eq:nu2_3}, \ref{eq:nu2_4}), we also require that:
\begin{align}
    u^2v + u + v \equiv 0 \pmod{N}, \label{eq:nu2p_1} \\ 
    u-pv \equiv 0 \pmod{N}. \label{eq:nu2p_2}
\end{align}

In Equation (\ref{eq:nu2p_1}), note that $v \equiv -u(u^2+1)^{-1} \pmod{N}$. We observe that $u^2 + 1$ is not invertible when $p \equiv 1 \pmod{4}$ and is always invertible when $p \equiv 3 \pmod{4}$. It follows that there are $p^r$ solutions when $p \equiv 3 \pmod{4}$ and $p^r - 2p^{r-1}$ solutions when $p \equiv 1 \pmod{4}$. When $p=2$, $u^2+1$ is not invertible modulo $N$ half of the time, so there are $2^r/2 = 2^{r-1}$ solutions. In Equation (\ref{eq:nu2p_2}), there are $p^{r-1}$ solutions, with no restrictions on $p$.

The identification from Lemma (\ref{lem:Csp_cosets}) identifies the two solutions from Equations (\ref{eq:nu2_1}, \ref{eq:nu2_2}), and each solution of Equations (\ref{eq:nu2p_1}, \ref{eq:nu2p_2}) is paired via
\[
\alpha(u,-u(u^2 + 1)^{-1}) \sim \begin{cases}
    \alpha(-u^{-1},u(u^2 + 1)^{-1}) & \text{if $u$ is invertible,} \\
    \beta(u, u/p) & \text{if $u$ is not invertible.}
\end{cases}
\]
This recovers the formula in \cite{DLM2022} for prime power levels $N = p^r$:
\[
\varepsilon_2^+ = \begin{cases}
    2^{r-1} & \text{ if $p = 2$,} \\
    1 + (p^r - p^{r-1})/2 & \text{ if $p \equiv 1 \pmod{4}$,} \\
    (p^r + p^{r-1})/2 & \text{ if $p \equiv 3 \pmod{4}$.}
\end{cases}
\]

    In the discussion after Lemma \ref{lem:Csp_cosets}, the involution $\omega$ acts on both $\GL_2(\Z/N\Z)/C_{\mathrm{sp}}(N)$ and
 $\prod_{p} \GL_2(\Z/p_i^{e_i}\Z)/C_{\mathrm{sp}}(p_i^{e_i})$, thus identifying pairs of solutions in $\prod_{p} \GL_2(\Z/p_i^{e_i}\Z)/C_{\mathrm{sp}}(p_i^{e_i})$. Dividing the above analysis into two parts, one coming from $C_{\mathrm{sp}}(N)$ and the other from $C_{\mathrm{sp}}^+(N)$, we obtain the formula for $\varepsilon_2^+$.

\end{proof}

Using the multiplicative relations from \cite{DLM2022}, one can produce formulas for $\varepsilon_2^*$, $\varepsilon_3^*$, and $\varepsilon_\infty^*$ of $X_{\mathrm{sp}}^*(N)$.

\begin{corollary}\label{cor:invariants_sp_star}
    Let $\varepsilon_2^*$, $\varepsilon_3^*$, $\varepsilon_\infty^*$ denote the number of elliptic points of orders $2$, $3$, and cusps in $X_{\mathrm{sp}}^*(N)$, respectively. Then we have that
    \begin{align*}
        \varepsilon_\infty^* &= \begin{cases}
        \frac{N}{3\cdot 2^{\omega(N)-2}}\prod\limits_{p\mid N}(1 + 1/p) &\text{if } 2||N, \\
        \frac{N}{2^{\omega(N)}}\prod\limits_{p\mid N}(1 + 1/p)&\text{otherwise,}
    \end{cases} \\
        \varepsilon_3^* &= \begin{cases}
        1 &  p \equiv 1\pmod{3} \text{ for all } p \mid N,   \\
        0 &\text{otherwise,}
    \end{cases}  \\
        \varepsilon_2^* &= 2^{\nu_2(N)-1} \cdot \prod_{\substack{p\mid N, \\ p \equiv 1 \bmod 4}} \left( 1 + \frac{p^{\nu_p(N)-1}(p-1)}{2}\right) \cdot \prod_{\substack{p\mid N, \\ p \equiv 3 \bmod 4}} \frac{p^{\nu_p(N)-1}(p+1)}{2}.
    \end{align*}
\end{corollary}

%%%%%%%%%%%%%%%%%%%%%%%%%%%%%%%%%%%%%%%%%%%%%%%%%%%%%%%%%%

\section{Non-split Cartan and normalizer of the non-split Cartan}\label{sec:nonsplit_cartan}

In this section, we further describe the modular curves $X_{\mathrm{ns}}(N)$, $X_{\mathrm{ns}}^+(N)$, and $X_{\mathrm{ns}}^*(N)$ from Section \ref{rem:Cartan_subgps} and Remark \ref{rem:plus_star}. We derive the invariants of their genus formulas for composite level $N$, using the methods outlined in \cite{BARAN20102753}.

%\asimina{Need to work on the beginning of this section, but the rest of it looks good!}

%The non-split Cartan $X_{\mathrm{ns}}(N)$ is the modular curve $X_H$ for $H\leq \GL_2(\widehat{\Z})$ given by the inverse image of a Cartan subgroup of $\GL_2(\Z/N\Z)$ that is non-split at every prime dividing $N$. %\asimina{Did we want to add a moduli description?}
%The normalizer of the non-split Cartan $X_{\mathrm{ns}}(N)$ is the modular curve $X_H$ for $H\leq \GL_2(\widehat{\Z})$ given by the inverse image of an extended Cartan subgroup of $\GL_2(\Z/N\Z)$ that is non-split at every prime dividing $N$. %\asimina{Did we want to add a moduli description?}

The non-split Cartan subgroup can also be defined explicitly. Let $R = \Z[\alpha]$ be a quadratic order such that $\alpha$ satisfies an irreducible monic polynomial $X^2 - uX + v \in \Z[X]$. Suppose that the discriminant of $R$ is coprime to $N$ and that every prime dividing $N$ is inert in $R$. Then $A = R/NR$ is free $\Z/N\Z$-module of rank $2$. For $x \in A$, the multiplication-by-$x$ map induces a map $A^\times \hookrightarrow \GL_2(\Z/N\Z)$ and the image is a non-split Cartan subgroup, denoted $C_{\mathrm{ns}}(N)$. 

For the normalizer of $C_{\mathrm{ns}}(N)$ in $\GL_2(\Z/N\Z)$, there is a ring automorphism $S_N$ of order $2$. Viewing $C_{\mathrm{ns}}(N)$ as $(\Z/N\Z)[\alpha]^\times$, since $\alpha$ is a root of $X^2 - uX + v$, $S_N(\alpha)$ also satisfies this polynomial and is therefore equal to $u-\alpha$. This gives an action of $S_N$ on $(\Z/N\Z)[\alpha]$ in the expected manner
\begin{align*}
1 &\mapsto 1 \pmod{N}, \\
\alpha &\mapsto u-\alpha \pmod{N}.
\end{align*}

The normalizer of the non-split Cartan subgroup, denoted $C_{\mathrm{ns}}^+(N)$, is defined as the group generated by the non-split Cartan subgroup and this involution $\langle C_{\mathrm{ns}}(N), S_N \rangle$.

We denote the modular curves associated to $C_{\mathrm{ns}}(N)$ and $C_{\mathrm{ns}}^+(N)$ by $X_{\mathrm{ns}}(N)$ and $X_{\mathrm{ns}}^+(N)$, respectively. There are natural finite morphisms $\phi_1$ and $\phi_2$ such that 
\[
X_{\mathrm{ns}}(N) \xrightarrow[]{\phi_1} X_{\mathrm{ns}}^+(N) \xrightarrow[]{\phi_2} X(1),
\]
where $\operatorname{deg}(\phi_1) = 2$ and $\operatorname{deg}(\phi_2) = N\cdot \varphi(N)/2$. Note that $\operatorname{deg}(\phi_2) = [\GL_2(\Z/N\Z): C_{\mathrm{ns}}^+(N)]$.

Let $C_{\mathrm{ns}} '(N) := C_{\mathrm{ns}}(N) \cap \SL_2(\Z/N\Z)$ and let $C_{\mathrm{ns}}^+{} '(N) := C_{\mathrm{ns}}^+ (N) \cap \SL_2(\Z/N\Z)$, where $C_{\mathrm{ns}} '(N)$ and $C_{\mathrm{ns}}^+{} '(N)$ can be identified with the following groups:
\begin{align*}
    C_{\mathrm{ns}}'(N) &= \{ M_y: y \in (\Z/N\Z)[\alpha]^\times, N(\alpha) = 1 \}, \\
    C_{\mathrm{ns}}^+ {} '(N) &= \{ M_y: y \in (\Z/N\Z)[\alpha]^\times, N(\alpha) = 1 \} \cup \{ M_y \circ S_N: y \in (\Z/N\Z)[\alpha]^\times, N(\alpha) = -1 \}, 
\end{align*}
where $M_y$ is the multiplication-by-$y$ map on $(\Z/N\Z)[\alpha]$, $N(x) = x\bar{x}$ is the norm map, and $\overline{a + b\alpha} = a + b(u - \alpha)$. For every element $a \in (\Z/N\Z)^\times$ or $a \in (\Z/N\Z)^\times /\pm 1$, choose $y_a \in (\Z/N\Z)[\alpha]^\times$ such that $N(y_a) = a$. We define the sets $\mathcal{Y}(N)$ and $\mathcal{Y}_{\pm}(N)$ as follows:
\begin{align*}
    \mathcal{Y}(N) &= \{ y_a \in (\Z/N\Z)[\alpha]^\times : N(y_a) = a, a \in (\Z/N\Z)^\times \}, \\
    \mathcal{Y}_{\pm} (N) &=\{ y_a \in (\Z/N\Z)[\alpha]^\times : N(y_a) = a, a \in (\Z/N\Z)^\times/\pm 1 \}.
\end{align*}

The following proposition provides a matrix representation for the cosets of $\Gamma_{\mathrm{ns}}$ and $\Gamma_{\mathrm{ns}}^+$ in $\SL_2(\Z)$.

\begin{prop}[Proposition 6.2 and 6.3, \cite{BARAN20102753}]\label{prop:matrix_rep}
    The coset representatives of $C_{\mathrm{ns}}'(N)$ $($resp. ${C_{\mathrm{ns}}^+} '(N))$ in $SL_2(\Z/N\Z)$ can be represented as linear maps that transform the basis $\{ 1, \alpha \}$ as 
\begin{align*}
    1 & \mapsto y^{-1}, \\
    \alpha & \mapsto \bar{y}(\alpha + x),
\end{align*}
where $x \in \Z/N\Z$ and $y \in \mathcal{Y}(N)$ $($resp. $y \in \mathcal{Y}_{\pm}(N))$.
\end{prop}

\begin{proof}
    We have the following equality
    \[
    [\SL_2(\Z) : \Gamma_{\mathrm{ns}}(N)] = [\SL_2(\Z/N\Z) : C_{\mathrm{ns}}'(N)] = [\GL_2(\Z/NZ) : C_{\mathrm{ns}}(N)]
    \]
    and a similar statement holds for $C_{\mathrm{ns}}^+(N)$. The number of cosets of $C_{\mathrm{ns}}(N)$ and $C_{\mathrm{ns}}^+(N)$ in $\GL_2(\Z/N\Z)$ are $N \varphi(N)$ and $N \varphi(N)/2$, respectively, and $ |\mathcal{Y}(N)| = \varphi(N)$ and $|\mathcal{Y}_\pm (N) |= \varphi(N)/2$. It remains to show that, under the identification above, the linear maps represent different cosets.
    
    Suppose we have two elements in the same coset of $\SL_2(\Z/N\Z)/C_{\mathrm{ns}}(N)$:
\begin{align*}
    1 &\mapsto y^{-1},  &1 &\mapsto y' {} ^{-1},\\
    \alpha &\mapsto \bar{y}(\alpha + x), &\alpha &\mapsto \overline{y'}(\alpha + x').
\end{align*}
Then we have $N(y) = N(y')$ and since $y,y' \in \mathcal{Y}(N)$, it follows that $y = y'$, which implies $x = x'$.

In the case of $C_{\mathrm{ns}}^+(N)$, observe that the linear maps that transform the basis $\{ 1,\alpha \}$ have two possibilities:
\begin{align*}
    1 &\mapsto y &1 &\mapsto \bar{y}\\
    \alpha &\mapsto y\alpha &\alpha &\mapsto \bar{y}\bar{ \alpha}.
\end{align*}
where $y \in (\Z/N\Z)[\alpha]^\times$ with $N(y) = 1$ and $N(y) = -1$, respectively. Suppose we have two elements in the same coset of $\SL_2(\Z/NZ)/C_{\mathrm{ns}}^+(N)$:
\begin{align*}
    1 &\mapsto y^{-1},  &1 &\mapsto y' {} ^{-1},\\
    \alpha &\mapsto \bar{y}(\alpha + x), &\alpha &\mapsto \overline{y'}(\alpha + x').
\end{align*}
Then the same argument shows that we have $N(y) = N(y')$ and since $y,y' \in \mathcal{Y}_\pm (N)$, it follows that $y = y'$, which implies $x = x'$.
\end{proof}

With this setup, we can begin calculating the quantities $\varepsilon_2, \varepsilon_3$, and $\varepsilon_\infty$ in the genus formula.

\begin{prop}\label{prop:norm_ns_infty}
    Let $\varepsilon_\infty$ and $\varepsilon_\infty^+$ denote the number of cusps in $X_{\mathrm{ns}}(N)$ and $X_{\mathrm{ns}}^+(N)$, respectively. Then we have that
    \begin{align*}
        \varepsilon_\infty &= \varphi(N), \\
        \varepsilon_\infty^+ &= \begin{cases}
            1 & \text{if } N = 2, \\
            \varphi(N)/2 & \text{otherwise.}
        \end{cases}
    \end{align*}
\end{prop}

\begin{proof}
        Note that $\SL_2(\Z)_\infty = \left\langle \begin{psmallmatrix} 1 & 1 \\ 0 & 1 \end{psmallmatrix} \right\rangle$. Let $g \in \SL_2(\Z)/\Gamma_{\mathrm{ns}}(N)$ such that $\begin{psmallmatrix} 1 & a \\ 0 & 1 \end{psmallmatrix} \in g^{-1} \Gamma_{\mathrm{ns}}(N) g$, for some $a \in \Z$, which stabilizes $\infty$. The element $g\begin{psmallmatrix} 1 & a \\ 0 & 1 \end{psmallmatrix}g^{-1} \in \Gamma_{\mathrm{ns}}(N)$ can be represented in two ways: one in the manner of Proposition \ref{prop:matrix_rep}, and the other as a multiplication-by-$k$ matrix, with $k \in (\Z/N\Z)[\alpha]^\times$, giving the following system of equations:
        \begin{align*}
            y^{-1} &\equiv y^{-1} k \pmod{N}, \\
            \bar{y}(\alpha + x) &\equiv \bar{y}(\alpha + x)k \pmod{N},
        \end{align*}
        where $x \in \Z/N\Z, y \in \mathcal{Y}(N), k \in (\Z/N\Z)[\alpha]^\times,$ and $N(k) = 1$. It follows that $k = 1$ and therefore $a = 0 \pmod{N}$. By \cite[Prop 1.37]{shimura1971}, the ramification index for each cusp of $X_{\mathrm{ns}}(N)$ is $N$ and since the degree of the morphism $X_{\mathrm{ns}}(N) \rightarrow X(1)$ is $N\varphi(N)$, we have $\varepsilon_\infty = \varphi(N)$.

        The same argument for $X_{\mathrm{ns}}^+ (N)$ shows that the ramification index for each cusp is $N$ and since the degree of the morphism $X_{\mathrm{ns}}^+(N) \rightarrow X(1)$ is $N\varphi(N)/2$, we have $\varepsilon_\infty^+ = \varphi(N)/2$. When $N = 2$, $\varepsilon_
        \infty^+ = 1$.
\end{proof}

\begin{prop}\label{prop:norm_ns_nu3}
    Let $\varepsilon_3$ and $\varepsilon_3^+$ denote the number of elliptic points of order $3$ in $X_{\mathrm{ns}}(N)$ and $X_{\mathrm{ns}}^+(N)$, respectively. Then we have that
    \begin{align*}
       \varepsilon_3 &= \begin{cases}
           2^{\omega(N)} &p \equiv 2 \pmod{3}  \text{ for all } \ p\mid N, \\
           0 & \text{otherwise,}
       \end{cases} \\
       \varepsilon_3^+ &= \varepsilon_3/2.
    \end{align*}
% \jun{remove $\omega$}
%         \begin{align*}
%         \varepsilon_3 &= 
%             \prod_{p\mid N} \left( 1 - \left( \frac{-3}{p} \right) \right),
% \\
%         \varepsilon_3^+ &= \varepsilon_3/2.
%     \end{align*}

\end{prop}

\begin{proof}
    Let $\rho = e^{2\pi i/3}$ be a third root of unity. Note that $\SL_2(\Z)_\rho = \left\langle \begin{psmallmatrix}
        0 & -1 \\ 1 & 1
    \end{psmallmatrix}\right\rangle$. By \cite[Prop 1.37]{shimura1971}, the elliptic points of order 3 on $X_{\mathrm{ns}}(N)$ are the points in the inverse image $f^{-1}(\rho)$, where $f: \HH^*/\Gamma_{\mathrm{ns}}(N) \rightarrow \HH^*/\SL_2(\Z)$, with ramification index 1. By the same proposition, there exists $g \in \SL_2(\Z)/\Gamma_{\mathrm{ns}}(N)$ such that $\SL_2(\Z)_\rho \subseteq g^{-1}\Gamma_{\mathrm{ns}}(N) g$. Therefore, $\varepsilon_3$ is equal to the number of coset representatives $g$ such that
    \[
    g\begin{pmatrix}
        0 & -1 \\ 1 & 1
    \end{pmatrix} g^{-1} = \gamma, 
    \]
    for some $\gamma \in \Gamma_{\mathrm{ns}}(N)$. By Proposition \ref{prop:matrix_rep}, we can represent $\gamma$ as a multiplication-by-$k$ matrix for some $k \in (\Z/N\Z)[\alpha]^\times$, and we obtain the following system of equations:
    \begin{align}
        \bar{y}(\alpha + x) &\equiv y^{-1} k \pmod{N},\label{eqn1} \\
        \bar{y}(\alpha + x) - y^{-1} &\equiv \bar{y}(\alpha + x)k \pmod{N},\label{eqn2}
    \end{align}
    where $x \in \Z/N\Z, y \in \mathcal{Y}(N), k \in (\Z/N\Z)[\alpha]^\times$, and $N(k) = 1$.
    
    We compute the number of solutions to the above equations, following \cite[Lem 7.7]{BARAN20102753}. Suppose that $x,y,k$ satisfy the equations above. Equation (\ref{eqn1}) can be rearranged as
    \[
    N(y)\alpha + N(y)x \equiv k \pmod{N}.
    \]
    Since $N(y) \in (\Z/N\Z)^\times$ and $k \not \in (\Z/N\Z)^\times$, equation (\ref{eqn1}) can be substituted into equation (\ref{eqn2}) as follows,
    \[
    y^{-1}k - y^{-1} \equiv y^{-1}k^2 \pmod{N},
    \]
    which implies that
    \begin{equation}\label{eqn3}
    k^2 - k + 1 \equiv 0 \pmod{N}.
    \end{equation}
    We can reduce equation \eqref{eqn3} modulo the primes $p$ dividing $N$. Hensel's lemma and the Chinese Remainder Theorem imply that there are two nontrivial solutions for each prime dividing $N$. Since $\{1, \alpha \}$ is a basis for $(\Z/N\Z)[\alpha]$, there are two unique solutions $(x,y) \in \Z/p^r \Z \times \mathcal{Y}(p^r)$ for the equation 
   \[
   N(y)\alpha + N(y)x \equiv k \pmod{p^r},
   \]
   for each prime $p \equiv 2 \pmod{3}$ dividing $N$ and $p^r || N$. If there is a prime $p \not\equiv 2 \pmod{3}$, then there will be no solutions. This yields the formula for $\varepsilon_3$. Since the morphism $X_{\mathrm{ns}}(N) \rightarrow X_{\mathrm{ns}}^+(N)$ is of degree $2$ and unramified over the elliptic points of order $3$, we obtain $\varepsilon_3^+ = \varepsilon_3/2$.
    \end{proof}

\begin{prop}\label{prop:norm_ns_nu2}
    Let $\varepsilon_2$ and $\varepsilon_2^+$ denote the number of elliptic points of order $2$ in $X_{\mathrm{ns}}(N)$ and $X_{\mathrm{ns}}^+(N)$, respectively. Then we have that 
%     \begin{align*}
%        \varepsilon_2 &= \begin{cases}
%            2^{\omega(N)} &p \equiv 3 \pmod{4} \text{ for all } \ p | N, \\
%            0 & \text{otherwise,}
%        \end{cases}\\
%     \varepsilon_2^+ &= \begin{cases}
%     \omega(N) &p \equiv 3\pmod{4} \text{ for all }  p|N,  \\
%        0 & \text{otherwise.}
%     \end{cases} \ + \ \left(\frac{1}{2}N \cdot \prod_{p|N}\left(1 + \frac{1}{p}\right) -
%          \#S\right).
%     \end{align*}
% \jun{remove $\omega$}
    \begin{align*}
        \varepsilon_2 &= \prod_{p\mid N} \left( 1 - \left( \frac{-1}{p} \right) \right),
\\
    \varepsilon_2^+ &= \sum_{p\mid N} \frac{1- \left( \frac{-1}{p} \right)}{2} \ + \ \left(\frac{1}{2}N \cdot \prod_{p\mid N}\left(1 + \frac{1}{p}\right) -
          \#S\right),
    \end{align*}
    where $S = \{a + b\alpha \in (\Z/N\Z)[\alpha]^\times/\pm 1: N(a+b\alpha) = -1, gcd(b,N) > 1\}$.
\end{prop}

\begin{proof}
    %As with the previous proof for $\varepsilon_3$ and $\varepsilon_3^+$, we 
    Note that $\SL_2(\Z)_i = \left\langle \begin{psmallmatrix}
        0 & -1 \\ 1 & 0
    \end{psmallmatrix}\right\rangle$ and the elliptic points of order 2 on $X_{\mathrm{ns}}(N)$ are the points in the inverse image $f^{-1}(i)$, where $f: \HH^*/\Gamma_{\mathrm{ns}}(N) \rightarrow \HH^*/\SL_2(\Z)$, with ramification index 1. Therefore, by \cite[Prop 1.37]{shimura1971}, $\varepsilon_2$ is equal to the number of coset representatives $g \in \SL_2(\Z)/\Gamma_{\mathrm{ns}}(N)$ such that the equality
    \[
    g\begin{pmatrix}
        0 & -1 \\ 1 & 0 
    \end{pmatrix} g^{-1}= \gamma,
    \]
    \noindent holds for some $\gamma \in \Gamma_{\mathrm{ns}}(N)$. Considering the above equality as linear maps using Proposition \ref{prop:matrix_rep}, we have the equations:
    \begin{align}
        \bar{y}(\alpha + x) &\equiv y^{-1}k \pmod{N}, \label{eqn4}\\
        -y^{-1} &\equiv \bar{y}(\alpha + x)k \pmod{N}, \label{eqn5}
    \end{align}
    where $x \in \Z/N\Z, y \in \mathcal{Y}(N), k \in (\Z/N\Z)[\alpha]^\times,$ and $N(k) = 1$. 

We compute the number of solutions to the above equations, following \cite[Lem 7.8]{BARAN20102753}. Suppose that $x,y,k$ satisfy the above equations. Equation \eqref{eqn4} can be rearranged and also substituted into equation \eqref{eqn5} to get the following:
\begin{align}
N(y)\alpha + N(y)x &\equiv k \pmod{N}, \label{eqn6} \\
-y^{-1} &\equiv y^{-1}k^2 \pmod{N}. \label{eqn7}
\end{align}
Since $N(y) \in (\Z/N\Z)^\times$, equation \eqref{eqn6} implies that $k \not \in \Z/N\Z$. From equation \eqref{eqn7}, we obtain
\[
-1 \equiv k^2 \pmod{N}.
\]
By the Chinese Remainder Theorem and Hensel's lemma, since $k$ is not an element of $(\Z/N\Z)^\times$ and hence of $(\Z/p\Z)^\times$ for any prime $p$ dividing $N$, we have $p \equiv 3 \pmod{4}$ for all primes $p$ dividing $N$, and there are two solutions for the system of equations. This proves the $\varepsilon_2$ part of the proposition.

Recall that,
\[
    C_{\mathrm{ns}}^+ {} '(N) = \{ M_y: y \in (\Z/N\Z)[\alpha]^\times, N(\alpha) = 1 \} \cup \{ M_y \circ S_N: y \in (\Z/N\Z)[\alpha]^\times, N(\alpha) = -1 \}.
\]
Therefore, for the normalizer of non-split Cartan subgroup, the same argument as above gives rise to two systems of equations, corresponding to the norm $+1$ and $-1$ elements:
\begin{align*}
    \bar{y}(\alpha + x) &\equiv y^{-1}k \pmod{N},  & \bar{y}(\alpha + x) &\equiv \bar{y}^{-1}k \pmod{N}, \\
    -y^{-1} &\equiv \bar{y}(\alpha + x)k \pmod{N}, &-y^{-1} &\equiv y(\bar{\alpha} + x)k \pmod{N}
\end{align*}
where $x \in \Z/N\Z, y \in \mathcal{Y}_{\pm}(N), k \in (\Z/N\Z)[\alpha]^\times$, and $N(k) = +1$ and $-1$, respectively.
% Notice that the system of equations on the left is the same as the $\varepsilon_2$ case. Since the covering map $X_{\mathrm{ns}}^+(N) \rightarrow X_{\mathrm{ns}}(N)$ is of degree $2$ and is unramified over these points, the norm $+1$ elements contribute $\omega(N)$ to $\varepsilon_2^+$ when the only primes dividing $N$ are congruent to $3 \pmod{4}$, and $0$ if there exists a prime congruent to $1 \pmod{4}$ dividing $N$.
Notice that the system of equations on the left are the same as the systems of equations in the $\varepsilon_2$ case. Since the covering map $X_{\mathrm{ns}}^+(N) \rightarrow X_{\mathrm{ns}}(N)$ is of degree $2$ and is unramified over these points, the norm $+1$ elements contribute $\sum_{p\mid N} \frac{1- \left( \frac{-1}{p} \right)}{2}$ to $\varepsilon_2^+$.

For the norm $-1$ elements, since $N(k) = -1$, the equations above on the right hand side are conjugate. Therefore, we want to know the number of solutions $(x,y) \in \Z/N\Z \times \mathcal{Y}_{\pm}(N)$ for which there exists a $k\in (\Z/N\Z)[\alpha]^\times$ such that the following equation holds
\[
\bar{y}(\alpha + x)\equiv \bar{y}^{-2}k \pmod{N}.
\]

Taking norms and rearranging terms, we obtain
\[
N(\alpha+x) \equiv -N(y)^{-2}\pmod{N}.
\]
Since the norm map $N: \mathcal{Y}_{\pm}(N) \rightarrow (\Z/N\Z)^\times/\pm 1$ is an isomorphism, we need to count the number of $x \in \Z/N\Z$ and $v \in (\Z/N\Z)^\times/\pm 1$ such that the following equality holds
\[
    N(\alpha+x) \equiv -v^{2}\pmod{N}.
\]

Now, we write $h = \frac{\alpha+x}{v} \in (\Z/N\Z)[\alpha]^\times/\pm1$, with $N(h) \equiv -1 \pmod{N}$. The problem at hand is to find how many such $h$'s exist. Note that the norm map
\[
N: (\Z/N\Z)[\alpha]^\times/\pm1 \rightarrow (\Z/N\Z)^\times
\]
is surjective and its kernel has order 
\[
\left(\frac{1}{2}N^2 \cdot \prod_{p\mid N}\left( 1-\frac{1}{p^2}\right)\right)/\varphi(N) = \frac{1}{2}\cdot \prod_{p\mid N} p^{r-1}(p+1).
\]

Therefore, the number of elements in $(\Z/N\Z)[\alpha]^\times/\pm1$ of norm $-1$ is equal to $\frac{1}{2}\prod_{p\mid N} p^{r-1}(p+1)$. Note that this count also includes elements whose $\alpha$-coefficient is not coprime to $N$. We exclude these elements in the next few paragraphs.

Let $\mathcal{P} \in \mathbb{P}(N):= \mathrm{PowerSet}(\{p: p \text{ prime, } p \mid N \})$ and define
\begin{align*}
\Theta_{\mathcal{P}}: (\Z/N\Z) [\alpha]^\times /\pm 1 &\rightarrow (\Z/N\Z)^\times \times \prod_{p \in \mathcal{P}} \frac{(\Z/p\Z) [\alpha]^\times}{(\Z/p\Z)^\times},\\
a + b\alpha &\mapsto (N(a+b\alpha), (a_p + b_p \alpha)_p),
\end{align*}
where $a_p,b_p$ are the representatives of $a,b$ under the reduction modulo $p$ and projection maps.

Then the image of $\Theta_{\mathcal{P}}$ is given by
\[
\operatorname{Im}(\Theta_{\mathcal{P}}) = \{ (b,(c_p)_p): b \equiv N(c_p)v_p^2 \pmod{p} \text{ for some } v_p \in (\Z/p\Z)^\times, \text{ for all } p \in \mathcal{P} \},
\]
which has cardinality $\varphi(N)\prod_{p \in \mathcal{P}} \frac{p+1}{2} \cdot 2^\delta$, where $\delta = 1$ when $2 \in \mathcal{P}$ and $ \delta = 0$ otherwise. The cardinality of the kernel can be calculated as 
\begin{align*}
\#\operatorname{ker}(\Theta_{\mathcal{P}}) &= \# C_{\mathrm{ns}}^+ (N) / \# \operatorname{Im}(\Theta_{\mathcal{P}}), \\
&= \frac{1}{2}N^2\left(\prod_{p\mid N} \left(1- \frac{1}{p^2}\right)\right)/\left(\varphi(N) \prod_{p\in \mathcal{P}} \frac{p+1}{2} \cdot 2^\delta\right), \\
&= 2^{\# \mathcal{P} -1 - \delta} N \cdot \prod_{p\mid N}\left(1 + \frac{1}{p}\right) / \prod_{p \in \mathcal{P}} (p+1).
\end{align*}

We are interested in $$\Theta^{-1}_{\mathcal{P}}(-1,(1,\ldots,1)) = \{a + b\alpha : N(a+b\alpha) = -1, \ \operatorname{gcd}(b,p) >1 , \text{ for all } p \in \mathcal{P} \}.$$

Note that this set has the same cardinality as the kernel, and if there exists $p \equiv 3 \pmod{4} \in \mathcal{P}$, then the preimage will be empty since the congruence condition 
\[
-1 \equiv 1\cdot v_p^2 \pmod{p}\]
does not have a solution when $p\equiv 3 \pmod{4}$. To summarize,
\[
\# \Theta_{\mathcal{P}}^{-1}(-1,(1,\ldots,1)) = \begin{cases} 
0 & \exists p \in \mathcal{P}, p \equiv 3 \pmod{4}, \\
2^{\# \mathcal{P} -1 - \delta} N \cdot \prod_{p\mid N}\left(1 + \frac{1}{p}\right) / \prod_{p \in \mathcal{P}} (p+1) & \text{otherwise.}
\end{cases}
\]

By an inclusion-exclusion argument, one could compute the cardinality of the set 
\[
S:=\{a + b\alpha: N(a+b\alpha) = -1, \ \operatorname{gcd}(b,N) > 1\}.
\]
Let $\mathbb{P}(N)_k := \{\mathcal{P} \in \mathbb{P}(N): \# \mathcal{P} = k \}$. Then,
\begin{align*}
    \# S = \sum_{k=1}^{\omega(N)} (-1)^{k+1} \left(\sum_{\mathcal{P} \in \mathbb{P}(N)_k} \#\Theta_{\mathcal{P}}^{-1}(-1,\underbrace{(1,\ldots,1)}_{k})\right)
\end{align*}
and we obtain
%\[
%    \varepsilon_2^+ = \underbrace{\begin{cases}
%    \omega(N) &p \equiv 3\pmod{4} \text{ for all } p|N,  \\
%        0 &\text{otherwise.}
%    \end{cases}}_{\text{norm} = +1} \ + \ \underbrace{\frac{1}{2}N\cdot \prod_{p|N}\left(1 + \frac{1}{p}\right) -
%          \#S}_{\text{norm} = -1}.
%\]
%\jun{remove $\omega$}
\[
    \varepsilon_2^+ = \underbrace{\sum_{p\mid N} \frac{1- \left( \frac{-1}{p} \right)}{2}
}_{\text{norm} = +1} \ + \ \underbrace{\frac{1}{2}N\cdot \prod_{p\mid N}\left(1 + \frac{1}{p}\right) -
          \#S}_{\text{norm} = -1}.
\]
\end{proof}

\begin{remark}
    Propositions \ref{prop:norm_ns_infty}, \ref{prop:norm_ns_nu3}, and \ref{prop:norm_ns_nu2} recover the formulas from \cite{BARAN20102753} when $N = p^r$. 
\end{remark}

Since the index, number of elliptic points and cusps are multiplicative in the level,  one can produce formulas for $\varepsilon_2^*$, $\varepsilon_3^*$, and $\varepsilon_\infty^*$ of $X_{\mathrm{ns}}^*(N)$ from prime power levels (see \cite{DLM2022}).

\begin{corollary}\label{cor:invariants_ns_star}
    Let $\varepsilon_2^*$, $\varepsilon_3^*$, $\varepsilon_\infty^*$ denote the number of elliptic points of orders $2$, $3$, and cusps in $X_{\mathrm{ns}}^*(N)$, respectively. Then we have that
    \begin{align*}
    \varepsilon_\infty^* &= \begin{cases}
        \varphi(N)/2^{\omega(N)-1} &\text{if } 2||N, \\
        \varphi(N)/2^{\omega(N)} &\text{otherwise,}
    \end{cases} \\
    \varepsilon_3^* &= \begin{cases}
    1 &  p \equiv 2\pmod{3} \text{ for all } p \mid N,  \\
        0 &\text{otherwise,}
    \end{cases} \\
    \varepsilon_2^* &= 2^{\nu_2(N)-1} \cdot \prod_{\substack{p\mid N, \\ p \equiv 1 \bmod 4}} \frac{p^{\nu_p(N)-1}(p-1)}{2} \cdot \prod_{\substack{p\mid N, \\ p \equiv 3 \bmod 4}} \left(1 + \frac{p^{\nu_p(N)-1}(p+1)}{2} \right).
    \end{align*}
\end{corollary}

%%%%%%%%%%%%%%%%%%%%%%%%%%%%%%%%%%%%%%%%%%%%%%%%%%%%%%%%%%

\section{Modular covers}\label{Xarith1}

An inclusion of open subgroups $H \leq G \leq \GL_2(\widehat{\Z})$ induces a covering map of modular curves $X_H \rightarrow X_G$. Modular covers are defined to be such morphisms. We can borrow tools from ramification theory and Riemann-Hurwitz to find the invariants of genus formulas of these curves. In this section, we describe the genus invariants of the modular curves $X_{\mathrm{arith},1}(M,MN)$ and $X_{\mathrm{arith},\pm 1}(M,MN)$, which are described in Section \ref{subsec:Xarith1}. 

The index of $\{ \begin{psmallmatrix}
    1 & M\ast \\ 0 & \ast
\end{psmallmatrix}\}$ in $\GL_2(\Z/MN\Z)$ is $M^3N^2 \prod_{p\mid MN} (1 - \frac{1}{p^2})$ and since $-I \not \in \{ \begin{psmallmatrix}
    1 & M\ast \\ 0 & \ast
\end{psmallmatrix}\}$, the $\PSL_2$-index of the corresponding subgroup is $\frac{1}{2} M^3N^2\prod_{p\mid MN} (1 - \frac{1}{p^2})$. Note that $X_{\mathrm{arith},1}(M,MN) \rightarrow X_1(MN)$ is a covering map of degree $M$. In particular, when $M=1$, the corresponding modular curve $X_{\mathrm{arith},1}(M,MN)$ is equal to $X_1(M)$, and when $N=1$, the corresponding modular curve is equal to $X_{\mathrm{arith}}(N)$.

\begin{prop}\label{prop:arith1_nu2_nu3}
    For $M,N >1$, there are no elliptic points of orders $2$ and $3$ for $\Gamma_{\mathrm{arith},1}(M,MN)$.
\end{prop}

\begin{proof}
    The elliptic points of orders $2$ and $3$ for $\SL_2(\Z)$ are $\SL_2(\Z)i$ and $\SL_2(\Z)e^{2\pi i/3}$, respectively. Any elliptic points of order $h 
    \in \{2,3\}$ of a congruence subgroup $\Gamma \subseteq \SL_2(\Z)$ must map to one of the above two points. Since $X_{\mathrm{arith},1}(M,MN) \rightarrow X_1(MN)$ is a covering map, and there are no elliptic points of order $h$ for $\Gamma_1(MN)$ when $M,N >1$, there are no elliptic points of order $h$ for $\Gamma_{\mathrm{arith},1}(M,MN)$.
\end{proof}

\begin{prop}\label{prop:xarith_nuinf}
    For $M,N >1$, let $\varepsilon_\infty$ denote the number of cusps in $X_{\mathrm{arith},1}(M,MN)$. Then we have that:
    \[
    \varepsilon_\infty = \begin{cases}
        1 & \text{if } M,N=1, \\
        2 & \text{if } M=1, \ N=2, \\
        3 & \text{if } M=1, \ N=4, \\
        \frac{1}{2}\cdot \sum\limits_{d \mid MN } \varphi \left(\frac{MN}{d}\right) \varphi(d) \gcd \left(M,\frac{MN}{d}\right) & \text{otherwise.}
    \end{cases}    
    \]
\end{prop}

\begin{proof}
    We follow a similar approach to that in \cite[Section 3.8]{Diamond2005}. Let $s'= a'/c', s=a/c \in \Q \cup \infty$. The coset decomposition $\Gamma_{\mathrm{arith},1}(M,MN) = \bigcup_{j=0}^{N-1} \Gamma(MN) \begin{psmallmatrix}
        1 & Mj \\ 0 & 1
    \end{psmallmatrix}$ gives the second equivalence: 
\begin{align*}
    \Gamma_{\mathrm{arith},1}(M,MN)s' = \Gamma_{\mathrm{arith},1}(M,MN)s &\iff s' \in \Gamma_{\mathrm{arith},1}(M,MN)s, \\
    &\iff s' \in \Gamma (MN) \begin{psmallmatrix}
        1 & Mj \\ 0 & 1
    \end{psmallmatrix} \text{ for some }j, \\
    &\iff \begin{psmallmatrix}
        a' \\ c'
    \end{psmallmatrix} \equiv \pm \begin{psmallmatrix}
        a + Mcj \\ c
    \end{psmallmatrix} \pmod{MN} \text{ for some }j.
\end{align*}    
This implies that the top row $a$ is determined modulo $\gcd(Mc,MN)$ and since this is a cusp, we also have $\gcd(a,c, MN) = 1$. Let $d := \gcd(c,MN)$, there are $\varphi(MN/d)$ elements $c$ such that $0 \leq c\leq MN-1$ and $\gcd(c,MN) = d$. The set 
\[\{ a: 0 \leq a \leq \gcd(Mc,MN), \gcd(a,c,MN) = 1\}  = \{ a: 0 \leq a \leq \gcd(Md,MN), \gcd(a,d) = 1\}\]
has size $\varphi(d)\gcd(M,MN/d)$. Summing over divisors of $MN$ yields the result.
\end{proof}

\begin{remark} \label{rem:arith1}
    Specializing to either $M=1$ or $N=1$ recovers the genus formulas for $X_1(N)$ and $X(M)$, respectively. 
\end{remark}

%\begin{itemize}
    %\item Proof read what we have so far (asimina - done!)
    %\item Work on Xarith1 and finish lmfdb pages
    %\item Add subsection/remark on quadratic refinements (Xpm1, Xarithpm1)
    %\item Review all formulas and check for omega's
    %\item \asimina{Once we have completed all of the above, we should send a draft to Pietro, Jeremy Rouse, Jen Balakrishnan, Drew Sutherland, John Voight, David Roe (and then post on Zulip)}
%\end{itemize}

%%%%%%%%%%%%%%%%%%%%%%%%%%%%%%%%%%%%%%%%%%%%%%%%%%%%%%%%%%

\begin{table}[h!]
    \renewcommand*{\arraystretch}{1.85}
    \caption{Invariants of genus formula}
    \label{tab:all_invariants}
    \begin{center}
    \begin{tabular}{c|c|c}
        \toprule
        $X_H$ & $i$ & $\varepsilon_2$ \\
        \midrule\midrule 
        $X_0(N)$ & $N \cdot \prod\limits_{\substack{p\mid N}}\left(1 + \frac{1}{p}\right)$ & $\begin{array}{l}0 \quad \text{if } N \equiv 0 \bmod 4 \\ \prod\limits_{\substack{p\mid N}}\left(1 + \left(\frac{-1}{p}\right)\right) \quad \text{otherwise}\end{array}$ \\
        \hline
        $\begin{array}{c}X_1(N), \\ X_{\pm 1}(N)\end{array}$ & $\begin{array}{l}1 \quad \text{if } N=1 \\ 3 \quad \text{if } N=2 \\ \frac{N^2}{2}\cdot \prod\limits_{\substack{p\mid N}}\left(1-\frac{1}{p^2}\right) \quad \text{otherwise}\end{array}$ & $\begin{array}{l}1 \quad \text{if } N=1,2 \\ 0 \quad \text{otherwise}\end{array}$ \\ 
        \hline
        $\begin{array}{c}X(N), \\ X_{\mathrm{arith}}(N)\end{array}$ & $\begin{array}{l}6 \quad \text{if } N=2 \\ \frac{N^3}{2} \cdot \prod\limits_{\substack{p\mid N}}\left(1 - \frac{1}{p^2}\right) \quad \text{if } N>2\end{array}$ & $\begin{array}{l}1 \quad \text{if } N=1 \\ 0 \quad \text{if } N>1\end{array}$ \\ 
        \hline
        $\begin{array}{c}X_{\mathrm{arith},1}(M,MN), \\ X_{\mathrm{arith},\pm 1}(M,MN)\end{array}$ & $\begin{array}{l}1 \quad \text{if } M,N=1 \\ 3 \quad \text{if } M=1, N=2 \\ 6 \quad \text{if } M = 2, N = 1 \\ \frac{M^3N^2}{2} \cdot \prod_{p|MN} \left(1 - \frac{1}{p^2}\right) \quad \text{otherwise}\end{array}$ & $\begin{array}{l}1 \quad \text{if } M,N=1 \\ 1 \quad \text{if } M=2, N=1 \\ 0 \quad \text{otherwise}\end{array}$ \\
        \hline
        $X_{\mathrm{sp}}(N)$ & $N^2 \cdot \prod\limits_{\substack{p\mid N}}\left(1 + \frac{1}{p}\right)$ & $\begin{array}{l}0 \quad \text{if } 2 \mid N, \\  \prod\limits_{\substack{p\mid N}} \left(1 + \left(\frac{-1}{p}\right) \right) \quad \text{otherwise}\end{array}$ \\
        \hline
        $X_{\mathrm{sp}}^+(N)$ & $\frac{N^2}{2} \cdot \prod\limits_{\substack{p\mid N}}\left(1 + \frac{1}{p}\right)$  & See Proposition \ref{prop:norm_sp_nu2}  \\ 
        \hline
        $X_{\mathrm{sp}}^*(N)$ & $\frac{N^2}{2^{\omega(N)}} \cdot \prod\limits_{\substack{p\mid N}}\left(1 + \frac{1}{p}\right)$ & See Corollary \ref{cor:invariants_sp_star} \\ 
        \hline
        $X_{\mathrm{ns}}(N)$ & $N\cdot \varphi(N)$ & $\begin{array}{l}\prod\limits_{\substack{p\mid N}} \left( 1- \left( \frac{-1}{p} \right) \right)
\end{array}$ \\
        \hline
        $X_{\mathrm{ns}}^+(N)$ & $\frac{1}{2}\cdot N\cdot \varphi(N)$ & See Proposition \ref{prop:norm_ns_nu2} \\
        \hline
        $X_{\mathrm{ns}}^*(N)$ & $\frac{N}{2^{\omega(N)}}\cdot \varphi(N)$ & See Corollary \ref{cor:invariants_ns_star} \\  
        \hline
        $X_{S_4}(p)$ & $\frac{1}{24}\cdot p(p^2 - 1)$ & $\frac{1}{4} \cdot \left( p - \left( \frac{-1}{p}\right) \right)$\\ 
        %\hline
        %$X_{\text{arith}}(N)$ & $\begin{array}{l}6 \quad \text{if } N=2 \\ \frac{N^3}{2} \prod\limits_{\substack{p\mid N}}\left(1 - \frac{1}{p^2}\right) \quad \text{if } N>2\end{array}$ & $\begin{array}{l}1 \quad \text{if } N=1 \\ 0 \quad \text{if } N>1\end{array}$ \\ 
        \bottomrule 
        \multicolumn{3}{c}{Table 1 continued on the next page \ldots} 
    \end{tabular}
    \end{center}
\end{table}

\begin{table}[h!]
    \renewcommand*{\arraystretch}{1.85}
    %\caption{Invariants of genus formula}
    \begin{center}
    \begin{tabular}{c|c|c}
    \multicolumn{3}{c}{Table 1 continued \ldots} \\
        \toprule
        $X_H$ & $\varepsilon_3$ & $\varepsilon_\infty$ \\
        \midrule\midrule 
        $X_0(N)$ & $\begin{array}{l}0 \quad \text{if } N \equiv 0 \bmod 9 \\ \prod\limits_{\substack{p\mid N}}\left(1 + \left(\frac{-3}{p}\right)\right) \quad \text{otherwise}\end{array}$ & $\sum\limits_{\substack{d\mid N, \\ d>0}} \varphi\left(\left(d,\frac{N}{d}\right)\right)$ \\
        \hline
        $\begin{array}{c}X_1(N), \\ X_{\pm 1}(N)\end{array}$ & $\begin{array}{l}1 \quad \text{if } N=1,3 \\ 0 \quad \text{otherwise}\end{array}$ & $\begin{array}{l}1 \quad \text{if } N=1 \\ 2 \quad \text{if } N=2 \\ 3 \quad \text{if } N=4 \\ \frac{1}{2}\cdot \sum\limits_{\substack{d\mid N}}\varphi(d)\varphi\left(\frac{N}{d}\right) \quad \text{otherwise}\end{array}$ \\ 
        \hline
        $\begin{array}{c}X(N), \\ X_{\mathrm{arith}}(N)\end{array}$ & $\begin{array}{l}1 \quad \text{if } N=1 \\ 0 \quad \text{if } N>1\end{array}$ & $\begin{array}{l}1 \quad \text{if } N=1 \\ \frac{i}{N} \quad \text{if } N>1\end{array}$ \\ 
        \hline
        $\begin{array}{c}X_{\mathrm{arith},1}(M,MN), \\ X_{\mathrm{arith},\pm 1}(M,MN)\end{array}$ & $\begin{array}{l}1 \quad \text{if } M=1 \text{ and } N=1,3  \\ 0 \quad \text{otherwise}\end{array}$ & See Proposition \ref{prop:xarith_nuinf} \\
        \hline
        $X_{\mathrm{sp}}(N)$ & $\begin{array}{l}0 \quad \text{if } 2 \mid N \text{ or } 3 \mid N, \\  \prod\limits_{\substack{p\mid N}}  \left(1 + \left(\frac{-3}{p}\right) \right) \quad \text{otherwise}\end{array}$ & $N \cdot \prod\limits_{\substack{p\mid N}}\left(1 + \frac{1}{p}\right)$ \\
        \hline
        $X_{\mathrm{sp}}^+(N)$ & $\begin{array}{l}0 \quad \text{if } 2 \mid N \text{ or } 3 \mid N, \\  \frac{1}{2} \cdot \prod\limits_{\substack{p\mid N}} \left(1 + \left(\frac{-3}{p}\right) \right) \quad \text{otherwise}\end{array}$ & $\frac{N}{2} \cdot \prod\limits_{\substack{p\mid N}}\left(1 + \frac{1}{p}\right)$ \\ 
        \hline
        $X_{\mathrm{sp}}^*(N)$ & See Corollary \ref{cor:invariants_sp_star} & See Corollary \ref{cor:invariants_sp_star} \\ 
        \hline
        $X_{\mathrm{ns}}(N)$ & $\begin{array}{l}           2^{\omega(N)} \quad \text{if } p \equiv 2 \pmod{3}  \text{ for all } \ p\mid N, \\
           0 \quad \text{otherwise,}
       \end{array}$ & $\varphi(N)$ \\
        \hline
        $X_{\mathrm{ns}}^+(N)$ & $\begin{array}{l}           2^{\omega(N)-1} \quad \text{if } p \equiv 2 \pmod{3}  \text{ for all } \ p\mid N, \\
           0 \quad \text{otherwise,}
       \end{array}$ & $\begin{array}{l}1 \quad \text{if } N = 2, \\   \frac{1}{2} \cdot \varphi(N) \quad \text{otherwise}\end{array}$

        \\
        \hline
        $X_{\mathrm{ns}}^*(N)$ & See Corollary \ref{cor:invariants_ns_star} & See Corollary \ref{cor:invariants_ns_star} \\  
        \hline
        $X_{S_4}(p)$ & $\frac{1}{3} \cdot \left( p - \left( \frac{-3}{p}\right) \right)$ & $\frac{1}{24}\cdot (p^2 - 1)$ \\ 
        %\hline
        %$X_{\text{arith}}(N)$ & $\begin{array}{l}1 \quad \text{if } N=1 \\ 0 \quad \text{if } N>1\end{array}$ & $\begin{array}{l}1 \quad \text{if } N=1 \\ \frac{i}{N} \quad \text{if } N>1\end{array}$ \\ 
         \bottomrule
    \end{tabular}
    \end{center}
    \label{tab:my_label}
\end{table}

\bibliography{bibliography}
\bibliographystyle{alpha}

%%%%%%%%%%%%%%%%%%%%%%%%%%%%%%%%%%%%%%%%%%%%%%%%%%%%%%%%%%

\end{document}